\documentclass[12pt]{amsart}

\usepackage{amssymb,ifthen,url}
\usepackage{algorithm,algorithmic,color,amsmath}
\usepackage{a4wide}
\usepackage{epsf}
\usepackage{centernot}   

\newtheorem{theorem}{Theorem}
\newtheorem{lemma}{Lemma}[section]

\newtheorem{proposition}[lemma]{Proposition}

\theoremstyle{definition}

\newtheorem{remark}[lemma]{Remark}

\numberwithin{equation}{section}

\newcommand {\NN}  {{\mathbb N}}

\newcommand {\ZZ}  {{\mathbb Z}}

\newcommand {\CC}  {{\mathbb C}}

\newcommand{\kommentar}[1]{}
\newcommand{\comment}[1]{}

\newcommand{\set}[1]{\left\{#1\right\}}

\newcommand{\rmus}[1]{\bigskip {\bf Remark for us: }{#1} \bigskip}

\DeclareMathOperator{\lcm}{ lcm \,}
\DeclareMathOperator{\Card}{ Card \,}

\def\eps{\varepsilon}

\def\ons{\setminus \set{0}}

\title{Salem numbers from a class of star-like trees}

\author[H.~Brunotte]{Horst~Brunotte}
\address{Haus-Endt-Strasse 88 \\ D-40593 D\"usseldorf, GERMANY}
\email{brunoth@web.de}

\author[J.~M.~Thuswaldner]{J\"org~M.~Thuswaldner}
\address{Chair of Mathematics and Statistics,
University of Leoben, Franz-Josef-Strasse 18, A-8700 Leoben,
AUSTRIA} \email{Joerg.Thuswaldner@unileoben.ac.at}

\thanks{This work was supported by the Agence Nationale de la Recherche and the Austrian Science Fund through the projects ``Fractals and Numeration'' (ANR-12-IS01-0002, FWF I1136) and ``Discrete Mathematics'' (FWF W1230).}

\date{\today}

\keywords{Salem number, Pisot number, Coxeter polynomial}

\subjclass[2010]{11K16}

\begin{document}

\begin{abstract}
We study the Coxeter polynomials associated with certain star-like trees. In particular, we exhibit large Salem factors of these polynomials and give convergence properties of their dominant roots.
\end{abstract}

\maketitle

\section{Introduction}

Several different methods to construct minimal polynomials of Salem numbers have been investigated in the literature (see e.g.\ \cite{BDGPS,boyd,grossmcmullen,smyth}). Various authors associate Salem numbers to Coxeter polynomials and use this relation in order to construct Salem numbers (cf.\ for instance \cite{cannwagr,floydplotnick,grohirmcm,lakatos,mckrowsmy}). In this paper we follow the very explicit approach of Gross et al.~\cite{grohirmcm} and provide precise information on the decomposition of Coxeter polynomials of certain star-like trees into irreducible factors, thereby giving estimates on the degree of the occurring Salem factor. 

To be more precise, let $r, a_0,\ldots, a_r \in \NN$ such that $a_0\ge 2,\ldots, a_r \ge 2$. We consider the star-like tree
 $T=T (a_0,\ldots, a_r)$ with $r+1$ arms of  $a_0-1,\ldots, a_r-1$ edges, respectively. According to \cite[Lemma 5]{mckrowsmy}
the Coxeter polynomial of $T (a_0,\ldots, a_r)$ is given by
$$
R_{T (a_0,\ldots, a_r)} (z) =\prod_{i=0}^{r} \Bigl( \frac{z^{a_i}-1}{z-1}\Bigr)\Bigl(z+1- z \sum_{i=0}^{r}  \frac{z^{a_i-1}-1}{z^{a_i}-1}\Bigr).
$$
Note that $R_T$ can be written as 
\begin{equation}\label{eq:decomp}
R_T(z)=C(z)S(z),
\end{equation}
where $C$ is a product of cyclotomic polynomials and $S$ is the minimal polynomial of a Salem number or of a quadratic Pisot number. Indeed, by the results of \cite{pena},  the zeros of $R_T$ are either real and positive or  have modulus 1. The decomposition \eqref{eq:decomp} now follows from \cite[Corollaries~7 and~9, together with the remark after the latter]{mckrowsmy}, as these results imply that $R_T$ has exactly one irrational real positive root of modulus greater than $1$. 

For Coxeter polynomials corresponding to star-like trees with three arms we are able to say much more about the factors of the decomposition \eqref{eq:decomp}. In particular, we shall prove the following result.

\begin{theorem}\label{thm1}
Let  $a_0,  a_1, a_2 \in \ZZ$  such that $a_2 > a_1 >a_0>1$ and
$(a_0, a_1, a_2)\ne (2,3,t)$ for all $t\in \set{4, 5, 6}$. Further, let $T:=T (a_0, a_1, a_2)$ be the star-like tree with three arms of  $a_0-1, a_1-1, a_2-1$ edges, and let $\lambda$ be its largest eigenvalue. Then $\tau >1$ defined by
$$
\sqrt {\tau}+ 1/\sqrt {\tau}=  \lambda
$$ 
is a Salem or a quadratic Pisot root of the Coxeter polynomial $R_{T}$ of  $T$. 
If $S$ is the minimal polynomial of $\tau$  then we can write 
\begin{equation}\label{eq:clear}
R_T(x)=S(x)C(x),
\end{equation}
where $C$ is a product of cyclotomic polynomials of orders bounded by $420 (a_2-a_1+a_0-1)$ whose roots have multiplicity bounded by an effectively computable constant $m(a_0,a_2-a_1)$. Thus 
\begin{equation}\label{SalemLowerBound}
\deg S \ge \deg R_T - m(a_0,a_2-a_1)\sum_{k\le 420 (a_2-a_1+a_0-1)}\varphi(k),
\end{equation}
where $\varphi$ denotes Euler's $\varphi$-function.
\end{theorem}

\begin{remark}[Periodicity properties of cyclotomic factors]\label{rem:period}
Gross et al.~\cite{grohirmcm} study certain Coxeter polynomials and prove periodicity properties of their cyclotomic factors. Contrary to their case, our Coxeter polynomials $R_T$ do not have the same strong separability properties (cf.\ Lemma~\ref{lem:mult}). For this reason, we could not exhibit analogous results for $C(x$), however, we obtain weaker periodicity properties in the following way.

In the setting of Theorem~\ref{thm1} assume that $a_0$ as well as $a_2-a_1$ are constant. For convenience set $S_{a_1}=R_{T(a_0,a_1,a_2)}$ and let $\zeta_k$ be a root of unity of order $k$. It follows from \eqref{expzdarst} below (see \eqref{pzdef} for the definition of $P$) that $S_{a_1}(\zeta_k)=0$ if and only if  $S_{a_1+k}(\zeta_k)=0$, i.e., the 
fact that the $k$-th cyclotomic polynomial divides $S_{a_1}$ depends only on the residue class of $a_1 \pmod k$. 
Therefore, setting $K:=\lcm\{1,2,\ldots,  420 (a_2-a_1+a_0-1)\}$,
the set of all cyclotomic polynomials dividing $S_{a_1}$ is determined by the residue class of $a_1 \pmod K$.
 
If we determine the set $\{k \,:\, k \le 420 (a_2-a_1+a_0-1),\, S_{a_1}(\zeta_k) = 0 \}$ for all $a_1 \le K$ we thus know exactly which cyclotomic factor divides which of the polynomials $S_{a_1}$ for $a_1\in\mathbb{N}$. Obviously, this knowledge would allow to improve the bound \eqref{SalemLowerBound}.
\end{remark}

\begin{remark}[Degrees of the Salem numbers]\label{rem:degree}
Theorem~\ref{thm1} enables us to exhibit Salem numbers of arbitrarily large degree.
 Indeed, if $a_0$ and the difference $a_2-a_1$ are kept small and  $a_1\to \infty$ then \eqref{SalemLowerBound} assures that $\deg S \to \infty$. We also mention here that Gross and McMullen \cite[Theorem~1.6]{grossmcmullen} showed that for any odd integer
$n\ge 3$ there exist infinitely many {\it unramified} Salem numbers of degree $2 n$;
recall that a Salem polynomial $f$ is said to be unramified if it satisfies $|f(-1)|=|f(1)|=1$.
The construction pursued in this work substantially differs from ours: it is proved that every unramified Salem polynomial arises from
an  automorphism of an indefinite lattice.
\end{remark}

If  two of the arms of the star-like tree under consideration get longer and longer, the associated Salem numbers converge to the $m$-bonacci number $\varphi_m$, where $m$ is the (fixed) length of the third arm. This is made precise in the following theorem.

\begin{theorem}\label{th:m}
Let $a_1 > a_0 \ge 2$ and $\eta \ge 1$ be given and set $a_2 = a_1+\eta$.  Then, for $a_1\to\infty$, the Salem root $\tau(a_0, a_1, a_2)$ of the  Coxeter polynomial associated with $T (a_0, a_1, a_2)$ converges to $\varphi_{a_0}$, where the degree of $\tau(a_0, a_1, a_2)$ is bounded from below by \eqref{SalemLowerBound}.
\end{theorem}

Besides that, we are able to give the following result which is valid for more general star-like trees.

\begin{theorem}\label{th:general}
Let $r\ge 1$, $a_r>  \cdots > a_1 > a_0 \ge 2$, and choose $k\in\{1,\ldots,r-1\}$.  Then, for fixed $a_0,\ldots,a_k$ and $a_{k+1},\ldots, a_r \to \infty$, the Salem root $\tau(a_0, \ldots , a_r)$ of the Coxeter polynomial associated with $T (a_0, \ldots, a_r)$ converges to the dominant Pisot root of
\begin{equation}\label{eq:qz}
Q(z)=(z+1-r+k)\prod_{i=0}^k(z^{a_i}-1)-z\sum_{i=0}^k(z^{a_i-1}-1)
\prod_{\begin{subarray}{c} j=0 \\ j\not=i \end{subarray}}^k (z^{a_j}-1).
\end{equation}
\end{theorem}

\section{Salem numbers generated by Coxeter polynomials of star-like trees}

For convenience, we introduce the polynomial
\begin{equation}\label{pzdef}
\begin{array}{rrl}
\displaystyle P(z)&:=&\displaystyle (z-1)^{r +1} R_{T (a_0,\ldots, a_r)} (z) \\[3pt]
&=&\displaystyle \Bigl( \prod_{i=0}^{r} (z^{a_i}-1)\Bigr) (z+1)- z \sum_{i=0}^{r}  \Bigl( (z^{a_i-1}-1) \prod_{j=0, j\ne i}^{r} (z^{a_j}-1)\Bigr).
\end{array}
\end{equation}
Of course, like $R_T$,  the polynomial $P$ can be decomposed as a product of a Salem (or quadratic Pisot) factor times a factor containing only cyclotomic polynomials.

Now, we concentrate on star-like trees with three arms, i.e., we assume that $r=2$. 

\begin{lemma} \label{pblock}
Let $a_2 > a_1 > a_0$. Then for $T(a_0,a_1,a_2)$ the polynomial $P(z)$ reads
\begin{equation}\label{expzdarst}
  P(z)= z^{a_1 + a_2} Q(z) + z^{a_1 + 1} R(z)+ S(z)
\end{equation}
with
\begin{eqnarray*}
 Q(z) &=& z^{a_0+1} - 2 z^{a_0}+ 1, \\
  R(z) &=& z^{a_2-a_1+a_0-1} - z^{a_2-a_1} + z^{a_0-1} -1,\\
  S(z) &=& -z^{a_0 +1} + 2z -1.
  \end{eqnarray*}
Moreover,
$$ 
\max\{ \deg(Q), \deg(R), \deg(S) \} = a_2-a_1+a_0-1, \quad \deg (P)= a_0 + a_1 +a_2 + 1,
$$ 
and the (naive) height of $P$ equals $2$.
\end{lemma}
\begin{proof}
This can easily be verified by direct computation.
\end{proof}

\begin{lemma}\label{lem:mult}
Let $a_2 > a_1 > a_0$ and let $P$ as in \eqref{pzdef} be associated to $T(a_0,a_1,a_2)$. Then there exists an effectively computable constant $m=m(a_0,a_2-a_1)$ which bounds the multiplicity of every root $z$ of $P$ with $|z|=1$. 
\end{lemma}

\begin{proof}
Observe, that $1$ is a root of $Q$, $R$, and $S$. Thus, for technical reasons, we work with $\tilde P(z)=P(z)/(z-1)$ and, defining $\tilde Q(z)$, $\tilde R(z)$, and $\tilde S(z)$ analogously, we write
\[
\tilde P(z)= z^{a_1 + a_2} \tilde Q(z) + z^{a_1 + 1} \tilde R(z)+ \tilde S(z).
\]
Our first goal is to bound the $n$-th derivatives $|\tilde P^{(n)}(z)|$ with $|z|=1 $ away from zero. To this end we define the quantities
\begin{align*}
\eta(a_0)&:=\min\{|\tilde Q(z)| \,:\, |z|=1 \} > 0, \\
E_n=E_n(a_0)&:=\max\{|\tilde Q^{(k)}(z)| \,:\, 1\le k\le n,\, |z|=1 \},  \\
F_0=F_0(a_0,a_2-a_1)&:=\max\{|\tilde R(z)| \,:\,  |z|=1 \}, \\
F_n=F_n(a_0,a_2-a_1,n)&:=\max\{|\tilde R^{(k)}(z)| \,:\, 1\le k\le n,\, |z|=1 \},  \\
G_n=G(a_0,n)&:=\max\{|\tilde S^{(n)}(z)| \,:\, |z|=1 \}.
\end{align*}
For $n\ge 1$ one easily computes that (note that $(x)_{n}=x(x-1)\cdots(x-n+1)$ denotes the Pochhammer symbol)
\begin{align*}
\tilde P^{(n)}(z) =& (a_1+a_2)_{(n)}\tilde Q(z)z^{a_1+a_2-n} +
(a_1+1)_{(n)}\tilde R(z)z^{a_1+1-n} \\
&+ \sum_{k=0}^{n-1}{n\choose k}(a_1+a_2)_{(k)}\tilde Q^{(n-k)}(z)z^{a_1+a_2-k} \\
&+ \sum_{k=0}^{n-1}{n\choose k}(a_1+1)_{(k)}\tilde R^{(n-k)}(z)z^{a_1+1-k} 
+ \tilde S^{(n)}(z).
\end{align*}
Now for  $|z|=1$ we estimate
\begin{align}
|\tilde P^{(n)}(z)| \ge & (a_1+a_2)_{(n)}\eta(a_0) - 2^{-n+1}(a_1+a_2)_{(n)}F_0 \nonumber \\
&- 2^{n-1}(a_1+a_2)_{(n-1)}E_n - 2^{n-1}(a_1+1)_{(n-1)}F_n - G_n\label{eq:betrag}\\
\ge& (a_1+a_2)_{(n)} \Big(\eta(a_0) - 2^{-n+1}F_0 - \frac{2^{n-1}(E_n+F_n)}{a_1+a_2-n+1} - \frac{G_n}{(a_1+a_2)_{(n)}}\Big). \nonumber
\end{align}
Now we fix $a_0$ and the difference $a_2-a_1$. Then we choose $n_0=n_0(a_0,a_2-a_1)$ such that 
$$
\eta(a_0) - 2^{-n_0+1}F_0 >0.
$$
In view of \eqref{eq:betrag} there exists a constant $c=c(a_0,a_2-a_1)$ such that for all $a_1,a_2$ with $a_1+a_2 > c$ (with our fixed difference) we have $|\tilde P^{(n_0)}(z)|  > 0$ for all $z$ with $|z|=1$. If, on the other hand, $a_1+a_2 \le c$, then we have $\deg \tilde P \le c+a_0$. Therefore, in any case, the multiplicity of a root of $\tilde P$ on the unit circle is bounded by $\max(n_0,c+a_0)$ and the result follows by taking $m=\max(n_0,c+a_0)+1$. 
\end{proof}

The following lemma is a simple special case of Mann's theorem.

\begin{lemma}\label{abcz}
Let   $a, b, c, p, q  \in \ZZ$ such that  $(p,q)\ne (0,0)$ and $a,b,c$ nonzero.
If $\zeta$ is  a  root of unity such that
$$a \zeta^p + b \zeta^q  +c= 0$$
then  the order of $\zeta$ divides $\; 6 \, \gcd (p, q).$
\end{lemma}

\begin{proof}
This is a special case of \cite[Theorem 1]{mann}.
\end{proof}

For subsequent use we recall some notation and facts (used in a similar context in \cite{grohirmcm}).
A  {\em divisor} on the complex plane is a finite sum
$$  D=\sum_{j  \in J}  a_j \cdot z_j$$
where $a_j \in \ZZ\ons $ and $${\rm supp \,} (D):= \set{z_j \in  \CC\;:\; j \in J}$$
is the support of $D$; $D$ is said to be  {\em effective} if all its coefficients are positive.

\medskip

The set of all divisors on $\CC$ forms the abelian group ${\rm Div \,} (\CC)$, and the natural
 evaluation map $\sigma: {\rm Div \,} (\CC) \to \CC $ is given by
$$ 
\sigma ( D) =\sum_{j  \in J}  a_j z_j\,.
$$

\medskip

A  {\em polar rational polygon}  (prp) is an effective divisor  $D=\sum_{j  \in J}  a_j \cdot z_j$ such that each $z_j$ is a root of unity and $ \sigma ( D) =0\,.$ In this case the order ${\rm o } (D) $ is the cardinality of the subgroup of
$ \CC \ons $ generated by the roots of unity $\set{z_j / z_k \;:\; j,k\in J}$.
The prp $D$ is called  {\em primitive} if there do not exist non-zero prp's $D'$ and $D''$ such that $D=D' + D''$.
In particular, the coefficients of $D', D''$ are positive, thus each prp can be expressed as a sum of primitive prp's.

\bigskip

Every polynomial  $f \in \ZZ[X] \ons$ can be uniquely written in the form
\begin{equation}\label{dcf}
  f=\sum_{j \in J} \eps_j a_j X^j
\end{equation}
with $J \subseteq \set{0, \ldots, \deg  (f)}$, $\eps_j=\pm 1$ and $a_j > 0$. We call
$$\ell (f):= \Card (J)$$
the length of $f$. For $\zeta \in \CC$ with $f(\zeta)=0$ we define the effective divisor of $f$ (w.r.t. $\zeta$) by
$$D f (\zeta):= \sum_{j \in J}  a_j (\eps_j \zeta^j) \,.$$

\bigskip

\begin{proposition}\label{farc}
Let $a_2 > a_1 > a_0$ and let $P$ as in \eqref{pzdef} be associated to $T(a_0,a_1,a_2)$. If   $\zeta$ is  a  root of unity such that  $P(\zeta) = 0$ then
the order of  $\zeta$ satisfies
$$
\mathrm{ord}(\zeta) \le 420 (a_2-a_1+a_0-1).
$$
\end{proposition}  
  
\begin{proof}
We follow the proof of \cite[Theorem 2.1]{grohirmcm} and write the polynomials $Q, R, S$
in the form
$$Q (X) =\sum Q_j (X), \qquad R (X) =\sum R_j (X), \qquad S (X) =\sum S_j (X)$$
with finite sums of integer polynomials such that
$$D P (\zeta)= \sum_{j}  D P_j (\zeta)= \sum_{j} \bigl( \zeta^{a_1 + a_2} D Q_j(\zeta) +  \zeta^{a_1 + 1} D R_j(\zeta)+  DS_j(\zeta)\bigr)$$
is a decomposition of the divisor  $ D P (\zeta)$ into primitive polar rational polygons $ D P_j (\zeta)$, thus for every $j$ the sum
$$\zeta^{a_1 + a_2} Q_j(\zeta) + \zeta^{a_1 + 1} R_j(\zeta)+ S_j(\zeta)$$
is the evaluation of the primitive prp $ D P_j (\zeta)$.
Observe that in view of Lemma \ref{pblock} the (naive) height of the polynomials $Q_j, R_j, S_j$ does not exceed  $2$ since the coefficients cannot increase when performing the decomposition of a prp into primitive prp's.

\medskip

Case 1: \qquad $\max \set{\ell (Q_j), \ell (R_j), \ell (S_j)} > 1$ for some $j$

\medskip

 Let us first assume $\ell (Q_j)> 1$ for some $j$. The ratio of any two roots of unity occurring in $D Q_j (\zeta)$
 can be written in the form $\pm \zeta^e$ with $1 \le e \le \deg (Q_j) \le \deg (Q)$. Therefore we have
 $$\frac{\mathrm{ord}(\zeta)}{2 \deg (Q) } \le {\rm o \,} (D P_j (\zeta)) .$$
 By Mann's Theorem \cite{mann}, $ {\rm o \,} (D P_j (\zeta))$ is bounded by the product of primes
 at most equal to
 $$\ell (P_j)\le \ell (Q_j) +\ell (R_j) + \ell (S_j)\le \ell (Q) +\ell (R) + \ell (S)\le 3+4+3=10\,.$$
 The product of the respective primes is at most $2\cdot 3 \cdot 5\cdot 7=210$. Therefore by Lemma \ref{pblock} we find
 $$\frac{\mathrm{ord}(\zeta)}{2 (a_0+1)}=\frac{\mathrm{ord}(\zeta)}{2 \deg (Q) } \le 210,$$
 which yields
 $$\mathrm{ord}(\zeta) \le 420 \; (a_0+1).$$
 Analogously, the other two cases yield
 $$\mathrm{ord}(\zeta) \le 420 \; (a_2-a_1+a_0-1) \qquad \text{ or } \qquad\mathrm{ord}(\zeta)\le 420 \; (a_0+1),$$
 and we conclude
\begin{equation}\label{resge1}
\mathrm{ord}(\zeta) \le 420 \; \max \set{a_0+1, \, a_2-a_1+a_0-1 } =420 \; (a_2-a_1+a_0-1).
\end{equation}

\medskip

Case 2:  \qquad $\max \set{\ell (Q_j), \ell (R_j), \ell (S_j)} \le 1$ for  all $j$

\medskip

In this case, $D P_j (\zeta)$ is either of the form
\begin{equation}\label{Pdouble}
D P_j (\zeta) = c_{j1}\zeta^{b_{j1}} + c_{j2}\zeta^{b_{j2}}  
\end{equation}
or of the form
\begin{equation}\label{Ptriple}
D P_j (\zeta) = c_{j1}\zeta^{b_{j1}} + c_{j2}\zeta^{b_{j2}}  + c_{j3}\zeta^{b_{j3}} , 
\end{equation}
where $c_{ji}\in \{-2,\ldots, 2\}$ by Lemma~\ref{lem:mult}. We distinguish two subcases.

\medskip

Case 2.1: There exists $j$ such that $D P_j (\zeta)$ is of the form \eqref{Ptriple}.

\medskip

 In this situation $D P_j (\zeta)$ can be written more explicitly as
\begin{equation}\label{Pt1}
D P_j (\zeta) = c_{j1}\zeta^{a_1+a_2+\eta_1} + c_{j2}\zeta^{a_1+\eta_2}  + c_{j3}\zeta^{\eta_{3}}
\end{equation}
or 
\begin{equation}\label{Pt2}
D P_j (\zeta) = c_{j1}\zeta^{a_1+a_2+\eta_1} + c_{j2}\zeta^{a_2+\eta_2}  + c_{j3}\zeta^{\eta_{3}},
\end{equation}
where $\eta_1\in\{0,a_0,a_0+1\}$, $\eta_2\in\{1,a_0\}$, and $\eta_3\in\{0,1,a_0+1\}$. If $D P_j (\zeta)$ is as in \eqref{Pt1} then $P_j (\zeta)=0$ implies that
\[
c_{j1}\zeta^{a_1+a_2+\eta_1-\eta_3} + c_{j2}\zeta^{a_1+\eta_2-\eta_3}  + c_{j3} =0.
\] 
Now, using Lemma~\ref{abcz} we gain
\[
\mathrm{ord}(\zeta) \mid 6 \mathrm{gcd}(a_1+a_2+\eta_1-\eta_3, a_1+\eta_2-\eta_3) 
\]
which yields
\[
\mathrm{ord}(\zeta) \mid 6(a_2-a_1+\eta_1-2\eta_2+\eta_3),
\]
hence, 
\begin{equation}\label{tripleresult}
\mathrm{ord}(\zeta) \le 6(2 a_0 + a_2 - a_1).
\end{equation}
If $D P_j (\zeta)$  is as in \eqref{Pt2}, by analogous arguments we again obtain \eqref{tripleresult}.

\medskip

Case 2.2:  For all $j$ the divisor $D P_j (\zeta)$ is of the form \eqref{Pdouble}.

\medskip

In this case we have to form pairs of the 10 summands of $D P(\zeta)$ to obtain the divisors $D P_j (\zeta)$. As $\ell(R)=\ell(S)=4$ there must exist $j_1,j_2$ such that $\ell(R_{j_1})=0$ and $\ell(S_{j_2})=0$.  In what follows, $c_{ij}\in\{-2,\ldots,2\}$, and  $\eta_1,\eta_1'\in\{0,a_0,a_0+1\}$, $\eta_2,\eta_2'\in\{1,a_0\}$, and $\eta_3,\eta_3'\in\{0,1,a_0+1\}$.
Then $D P_{j_1} (\zeta)$ is of the form
\begin{equation}\label{pairj1}
D P_{j_1} (\zeta) = c_{j_11}\zeta^{a_1+a_2+\eta_1} + c_{j_13}\zeta^{\eta_{3}}
\end{equation}
which yields
\begin{equation*}\label{pairj1eq}
c_{j_11}\zeta^{a_1+a_2+\eta_1-\eta_3} + c_{j_13} = 0,
\end{equation*}
and, hence, 
\begin{equation}\label{pairj1ord}
\mathrm{ord}(\zeta)\mid 2(a_1+a_2+\eta_1-\eta_3).
\end{equation}
For $D P_{j_2} (\zeta)$ we have two possibilities.  Either we have 
\begin{equation}\label{pairj21}
D P_{j_2} (\zeta) = c_{j_21}\zeta^{a_1+a_2+\eta'_1} + c_{j_22}\zeta^{a_1+\eta'_2} .
\end{equation}
This yields
\begin{equation*}\label{pairj21eq}
c_{j_21}\zeta^{a_2+\eta'_1-\eta_2'} + c_{j_22} = 0,
\end{equation*}
and, hence, 
\begin{equation}\label{pairj21ord}
\mathrm{ord}(\zeta)\mid 2(a_2+\eta_1'-\eta_2').
\end{equation}
The second alternative for $D P_{j_2} (\zeta)$ reads
\begin{equation}\label{pairj22}
D P_{j_2} (\zeta) = c_{j_21}\zeta^{a_1+a_2+\eta'_1} + c_{j_22}\zeta^{a_2+\eta'_2}.
\end{equation}
This yields
\begin{equation*}\label{pairj22eq}
c_{j_21}\zeta^{a_1+\eta'_1-\eta_2'} + c_{j_22} = 0,
\end{equation*}
and, hence, 
\begin{equation}\label{pairj22ord}
\mathrm{ord}(\zeta)\mid 2(a_1+\eta_1'-\eta_2').
\end{equation}
If $P_{j_2}$ is of the form \eqref{pairj21}, then \eqref{pairj1ord} and \eqref{pairj21ord} yield
that
\[
\mathrm{ord}(\zeta)\mid 2 \mathrm{gcd}(a_1+a_2+\eta_1-\eta_3, a_2+\eta_1'-\eta_2'),
\]
hence, $\mathrm{ord}(\zeta)\mid 2(a_1-a_2+\eta_1-2\eta_1'+2\eta_2'-\eta_3)$ and therefore
\begin{equation}\label{endres2}
\mathrm{ord}(\zeta)\le 2(a_2-a_1+3a_0+1).
\end{equation}
If $P_{j_2}$ is of the form \eqref{pairj22}, then \eqref{pairj1ord} and \eqref{pairj22ord} yield
\[
\mathrm{ord}(\zeta)\mid 2 \mathrm{gcd}(a_1+a_2+\eta_1-\eta_3, a_1+\eta_1'-\eta_2'),
\]
hence, \eqref{endres2} follows again.

Summing up, the proposition is proved by combining  \eqref{resge1}, \eqref{tripleresult}, and \eqref{endres2}.

%
%
%
%
%
%
%
%
%
%
%
%
%
%
%
%
%
%
 \end{proof}

Combining results from \cite{mckrowsmy} with our previous considerations we can now prove the main theorem of this paper.

\begin{proof}[Proof of Theorem~\ref{thm1}]
The fact that $\tau$ is either a Salem number or a quadratic Pisot number  as well as the decomposition of $R_T$ given in \eqref{eq:clear} follows immediately from~\eqref{eq:decomp}. The bound on the orders of the roots of the cyclotomic polynomials is a consequence of Proposition~\ref{farc}. Together with Lemma~\ref{lem:mult} this proposition yields the estimate $\eqref{SalemLowerBound}$ on the degree of $S$, where the explicitly computable constant $m(a_0,a_2-a_1)$ is the one stated in Lemma~\ref{lem:mult}.
 \end{proof}

\section{Convergence properties of Salem numbers generated by star-like trees}

In this section we prove Theorems~\ref{th:m} and~\ref{th:general}.
In the following proof of Theorem~\ref{th:m} we denote by $M_m(x)=x^m-x^{m-1}-\dots-x-1$ the minimal polynomial of the $m$-bonacci number $\varphi_m$.

\begin{proof}[Proof of Theorem~\ref{th:m}]
Note that $Q(x)=(x-1)M_{a_0}(x)$ holds. The theorem is proved if for each $\varepsilon > 0$ the polynomial $P(x)$ has a root $\zeta$ in the open ball $B_\varepsilon(\varphi_{a_0})$ for all sufficiently large $a_1$. We prove this by using Rouch\'e's Theorem. Let $\varepsilon > 0$ be sufficiently small and set $C_\varepsilon : =\partial B_\varepsilon(\varphi_{a_0})$. Then $\delta := \min\{|M_{a_0}(x)|\,:\, x \in C_\varepsilon \} >0$. Thus, on $C_\varepsilon$ we have the following estimations.
\begin{align}
\label{est}
|x^{a_1+1}R(x)+S(x)| &< (\varphi_{a_0}+\varepsilon)^{a_1+1}4(\varphi_{a_0}+\varepsilon)^{\eta+a_0-1} + 4(\varphi_{a_0}+\varepsilon)^{a_0+1},
\\
|x^{a_1+a_2}Q(x)| &> (\varphi_{a_0}-\varepsilon)^{2a_1+\eta}(\varphi_{a_0}-1-\varepsilon)\delta.
\nonumber
\end{align}  
Combining these two inequalities yields
\begin{equation}\label{P_est}
|P(x)| > (\varphi_{a_0}-\varepsilon)^{2a_1+\eta}(\varphi_{a_0}-1-\varepsilon)\delta - 
4(\varphi_{a_0}+\varepsilon)^{a_1+a_0+\eta} - 4(\varphi_{a_0}+\varepsilon)^{a_0+1}.
\end{equation}
Since \eqref{est} and \eqref{P_est} imply that for sufficiently large $a_1$ we have
\[
|P(x)| > |x^{a_1+1}R(x)+S(x)| \qquad (x\in C_\varepsilon),
\]
Rouch\'e's Theorem yields that $P(x)$ and $x^{a_1+a_2}(x-1)M_{a_0}(x)$ have the same number of roots in $B_\varepsilon(\varphi_{a_0})$. Thus our assertion is proved.
\end{proof}

To prove Theorem~\ref{th:general} we need the following auxiliary lemma.

\begin{lemma}\label{lem:aux}
Let $r\ge 1$, $a_r>  \cdots > a_1 > a_0 \ge 2$, and choose $k\in\{1,\ldots,r-1\}$. If $P$ is as in \eqref{pzdef} and $Q$ as in \eqref{eq:qz} then, for fixed $a_0,\ldots,a_k$ and $a_{k+1},\ldots, a_r$ sufficiently large, we have
\[
\#\{\xi \in \mathbb{C}\;:\, P(\xi)=0,\, |\xi| > 1\} = \#\{\xi \in \mathbb{C}\;:\, Q(\xi)=0,\, |\xi| > 1\} = 1. 
\]
\end{lemma}

\begin{proof}
First observe that 
\begin{equation}\label{eq:PQesti}
P(z)= z^{a_{k+1} + \cdots + a_r}Q(z) + O(z^{a_{k+2} + \cdots + a_r+\eta})
\end{equation}
for some fixed constant $\eta \in \mathbb{N}$.Since by ~\eqref{eq:decomp} the polynomial $P$ has exactly one (Salem) root outside the unit disk, it is sufficient to prove the first equality in the statement of the lemma. 

We first show that $Q$ has at least one root $\xi$ with $|\xi| > 1$. This is certainly true for $k < r-2$ as in this case we have $|Q(0)|>1$. For $k\in \{r-2,r-1\}$ we see that
\[
Q^{(\ell)}(1)=0,\;\hbox{for} \;  0\le \ell < a_0+\cdots+a_k - 1, \quad Q^{(a_0+\cdots+a_k - 1)}(1)<0.
\] 
As the leading coefficient of $Q$ is positive, this implies that $Q(\xi)=0$ for some $\xi >1$.

Now we show that $Q$ has at most one root $\xi$ with $|\xi| > 1$. Assume on the contrary that there exist two distinct roots $\xi_1,\xi_2\in\mathbb{C}$ of $Q$ outside the closed unit circle. Applying Rouch\'e's Theorem to $P(z)$ and $z^{a_{k+1}+\cdots+a_r}Q(z)$ shows by using \eqref{eq:PQesti} that also $P$ has two zeroes outside the closed unit circle which contradicts the fact that $P$ is a product of a Salem polynomial and cyclotomic polynomials, see~\eqref{eq:decomp}.
\end{proof}

\begin{proof}[Proof of Theorem~\ref{th:general}]
From Lemma~\ref{lem:aux} and \eqref{eq:qz} we derive that
\[
Q(z)=C(z)T(z)z^s,
\]
where $s\in\{0,1\}$ and $T$ is a Pisot or a Salem polynomial. To show the theorem we have to prove that $T$ is a Pisot polynomial. It suffices to show that $C(z)T(z)$ is not self-reciprocal. 

We distinguish three cases. If $k < r-2$ then $|C(0)T(0)|>1$, hence, as this polynomial has leading coefficient $1$ it cannot be self-reciprocal.

Denote the $\ell$-th coefficient of the polynomial $f$ by $[z^\ell]f(z)$. If $k=r-2$ we have $[z]C(z)T(z)=2(-1)^{r-1}$ and $[z^{a_0+\cdots+a_{r-2}}]C(z)T(z)=-r$. As $k>0$ we have $r>2$ and again the polynomial cannot be self-reciprocal.

Finally for $k=r-1$ we have $[z]C(z)T(z)=0$ and $[z^{a_0+\cdots+a_{r-1}-1}]C(z)T(z)=-r$ which again excludes self-reciprocity of $C(z)T(z)$.
\end{proof}

\comment{
\rmus{Finding cyclotomic parts in polynomials:   \cite[Section 2]{beukerssmyth}}

\rmus{Modifying and applying techniques developped in  \cite[Section~7]{grossmcmullen}?}
}

\section{Concluding remarks}

In this note we have studied Coxeter polynomials of star-like trees with special emphasis on star-like trees with three arms. It would be nice to extend Theorem~\ref{th:m} to star-like trees $T(a_0,\ldots, a_r)$ with four and more arms in order to get lower estimates on the degrees of the Salem polynomials involved in Theorem~\ref{th:general}. In fact, the estimate on the maximal multiplicity of the irreducible factors of $R_T$ contained in Lemma~\ref{lem:mult} can be carried over to star-like trees with larger values of $r$. Concerning Proposition~\ref{farc}, the argument based on Mann's Theorem used in order to settle Case~1 of its proof can be extended to $T(a_0,\ldots, a_r)$, however, in the situation of Case~2 we were not able to prove that the orders of the occurring roots of unity are bounded by a reasonable bound. We expect that a generalization of this case requires new ideas.

\bigskip

{\bf Acknowledgment}. 
The authors are indebted to the referee for insightful comments on the first version
of this paper.


\def\cprime{$'$}

\end{document}